\documentclass[preprint,12pt]{elsarticle_arxiv}

\usepackage{amssymb}
\usepackage{amsthm}

\usepackage{amsfonts}
\usepackage{epsfig}
\usepackage{subfigure}

\newtheorem{thm}{Theorem}[section]
\newtheorem{lem}[thm]{Lemma}

\journal{ }

\begin{document}

\begin{frontmatter}



\title{An Efficient Numerical Algorithm for the $L^{2}$ Optimal Transport Problem with
Applications to Image processing}


\author{Louis-Philippe Saumier}\ead{lsaumier@uvic.ca}
\author{Martial Agueh}\ead{agueh@uvic.ca}
\author{Boualem Khouider}\ead{khouider@uvic.ca}
\address{Department of Mathematics and Statistics, University of Victoria, 
PO BOX. 3060 STN, Victoria, B.C., V8W 3R4, Canada. }

% \thanks{Mathematics and Statistics department, University of
% Victoria ({\tt }).}
%         \and Martial Agueh\thanks{Ibid. ({\tt agueh@uvic.ca}).}
%         \and Boualem Khouider\thanks{Ibid. ({\tt khouider@uvic.ca}).}}
% 
% \address{}

\begin{abstract}
We present a numerical method to solve the optimal transport problem with a quadratic cost when the source and target measures are periodic probability 
densities. This method is based on a numerical resolution of the corresponding Monge-Amp\`ere equation. We extend the damped Newton algorithm of Loeper and Rapetti \cite{LR} to the more general case of a non uniform density which is relevant to the optimal transport problem, and we show that  our algorithm converges for sufficiently large damping coefficients.  The main idea consists of designing an iterative scheme where the fully nonlinear equation is approximated by a non-constant coefficient linear elliptic  PDE that we solve numerically. We introduce several improvements and some new techniques  for the numerical resolution of the corresponding linear system. Namely, we use a Fast Fourier Transform (FFT) method by Strain \cite{St}, which allows to increase the efficiency of our algorithm against the standard finite difference method. Moreover, we use a fourth order finite difference scheme to approximate the partial derivatives involved in the nonlinear terms of the Newton algorithm, which are evaluated once at each iteration; this leads to a significant improvement of the accuracy of the method, but does not sacrifice its efficiency. Finally, we present some numerical experiments which demonstrate the robustness and efficiency of our method on several examples of image processing, including  an application to multiple sclerosis disease detection.
\end{abstract}


\end{frontmatter}

% \linenumbers

\section{Introduction}\label{sec_intro}

The optimal transport problem, also known as the Monge-Kantorovich problem,
originated from a famous engineering problem by Monge \cite{M} for which Kantorovich
produced a more tractable relaxed formulation \cite{Kan}. This problem
deals with the optimal way of allocating resources from one site to another while
keeping the cost of transportation minimal. Formally, if $\mu$ is a probability
measure modelling the distribution of material in the source domain $X\subset \mathbb{R}^d$,
and $\nu$ is another probability measure modelling the structure of the target domain $Y\subset\mathbb{R}^d$, the
Monge-Kantorovich problem consists of finding the optimal transport plan $T$ in 
\begin{equation}\label{Mongeproblem}
\inf_{T:X\rightarrow Y}\Big\{ \int_X c\left(x-T(x)\right)\,\mbox{d}\mu(x); \;T_{\#}\mu=\nu\Big\},
\end{equation}
where $c(x-y)$ denotes the cost of transporting a unit mass of
material from a position $x\in X$ to a location $y\in Y$, and $T_{\#}\mu=\nu$ means
that $\nu(B)=\mu\left(T^{-1}(B)\right)$ for all Borel sets $Y\subset\mathbb{R}^n$, that is,
the quantity of material supplied in a region $B$ of $Y$ coincides with the total
amount of material transported from the region $T^{-1}(B)$ of $X$ via the transport plan $T$.
When the cost function is quadratic, i.e. $c(x-y)=|x-y|^2/2$, the corresponding
optimal transport problem in known as the $L^2$ optimal transport problem. This
particular case has attracted many researchers in the past few decades, and a lot of
interesting theoretical results have been obtained along with several applications 
in science and engineering, such as meteorology, fluid dynamics and mathematical
economics. We refer to the recent monographs of Villani \cite{V, Vi} for an
account on these developments. One of the most important results concerns the form of the
solution to the $L^2$ optimal transport problem. Indeed if both the source and
target measures are absolutely continuous with respect to the Lebesgue measure on
$\mathbb{R}^d$,  $\mbox{d}\mu/\mbox{d}x = f(x)$, $\mbox{d}\nu/\mbox{d}x = g(x)$,  Brenier
\cite{B} showed that the $L^2$ optimal transport problem has a unique invertible
solution $\widetilde{T}$ ($\mu$ a.e.) that is characterized  by the gradient of a convex
function, $\widetilde{T}=\nabla\Psi$. Moreover, if $f$ and $g$ are sufficiently regular (in a
sense to be specified later), it is proved that $\Psi$ is of class $C^2$ and
satisfies the Monge-Amp\`ere equation 
\begin{equation}\label{MongeAmpere}
g(\nabla \Psi(x))\det{(D^{2}\Psi(x))}=f(x),
\end{equation}
(see Delano\"e \cite{De}, Caffarelli \cite{Ca1,Ca2}, Urbas \cite{Ur}). Therefore,
for smooth source and target probability densities $f$ and $g$, a convex solution
$\Psi$ to the Monge-Amp\`ere equation (\ref{MongeAmpere}) provides the optimal
solution $\widetilde{T}=\nabla\Psi$ to the $L^2$ optimal transport problem.

In this paper, we are interested in the numerical resolution of the $L^2$ optimal
transport problem. Concerning this issue, only few numerical methods are available
in the literature, e.g. \cite{BB,CW,HRT}. Even if some of these methods are efficient, they all have 
issues that call for improvement, most of the time regarding their convergence (which is 
not always guaranteed). Therefore, the numerical results they produce are sometimes not
satisfactory. Although this list is not exhaustive, a more elaborate discussion on the advantages and disadvantages of each of these methods is given in the concluding remarks in Section
 \ref{sec_conc}. Our goal in this paper is to present an efficient numerical method which is
 stable and guaranteed to converge (before discretization). For that, contrarily to the previously existing methods, we propose to solve
numerically the Monge-Amp\`ere equation (\ref{MongeAmpere}) for a convex solution
$\Psi$, then producing the solution $\widetilde{T}=\nabla\Psi$ to the $L^2$ optimal transport
problem. The numerical resolution of the Monge-Amp\`ere equation is still a challenge, though some progress has been made recently, e.g. \cite{LR,O}.  In \cite{LR}, Loeper and Rapetti considered the Monge-Amp\`ere equation (\ref{MongeAmpere}) in the particular case where the target density
$g$ is uniform, $g=1$, and used a damped Newton algorithm to solve the equation. They also provided a proof of convergence of their method (before discretization) 
under some appropriate regularity assumptions  on the initial density. However, their assumption on the final density, (i.e. $g=1$), is too restrictive 
and therefore strongly limits the potential applications of their result, from an optimal transport point of view. Here, we extend their method to the general case where the final density $g$ is arbitrary (but sufficiently regular), which is the general context of the optimal transport problem, and we make their algorithm more efficient by presenting several numerical improvements. This is a novelty in our work compared to \cite{LR}.
Specifically, we approximate the fully nonlinear PDE (\ref{MongeAmpere})
by a sequence of linear elliptic PDE's via a damped Newton algorithm. Then we extend the convergence result of \cite{LR} to the more general context where the final density $g$ is arbitrary (under some suitable regularity assumptions). Here several new techniques are introduced due to the difficulties arising from final density $g$ which is no more uniform in our work. Moreover, we present several numerical improvements to the method introduced in \cite{LR}.  More precisely, we solve the linear PDE's approximating (\ref{MongeAmpere}) using two different discretization methods, namely, a standard
second order finite difference implementation, used in \cite{LR},  and a fast Fourier transform (FFT) implementation (see Strain \cite{St}). 
The FFT algorithm provides what appears to be a globally stable $\mathcal{O}(P{\rm
Log}P)$ method. In addition to this FFT speed up compared to \cite{LR}, we also use for both implementations
fourth order centered differences to approximate the first and second order derivatives involved in the nonlinear
right-hand side of the Newton algorithm. To prove the theoretical convergence of the solutions of the
linear elliptic PDE's to the actual convex solution of the Monge-Amp\`ere equation
(\ref{MongeAmpere}), we exploit interior a priori estimates on the classical solutions
to the Monge-Amp\`ere equation. As far as we know, no global estimates that we could use are available
for this equation. We thus restrict ourselves to a periodic setting to take advantage of the local estimates. 
Even with this restriction, our numerical method still gives in practice very good
results when applied to a wide range of examples, periodic or non-periodic (see
Section \ref{sec_exp}).

The paper is organized as follows. In Section \ref{sec_transp} we present the problem together with
 some background results that will be used later. 
In Section \ref{sec_algo}, we introduce the damped Newton algorithm for the Monge-Amp\`ere equation in 
the general case of the $L^{2}$ optimal transport problem and discuss its convergence. 
In Section \ref{sec_disc}, we propose two different ways of
discretizing the algorithm, and then test these implementations on three examples in
Section \ref{sec_exp}. One of these examples is taken from Medical
imaging, namely, the detection of the multiple sclerosis (MS) disease in a brain magnetic resonance imaging (MRI) scan. 
Finally, we conclude with some remarks in Section \ref{sec_conc}.

\section{Problem setting}\label{sec_transp}
In what follows, $d\geq 1$ is an integer, and we denote by $e_i$ the
$i^{th}$-canonical unit vector of $\mathbb{R}^d$. A function
$\zeta:\mathbb{R}^{d}\rightarrow\mathbb{R}$ is said to be $1$-periodic if
$\zeta(x+e_{i})=\zeta(x)$ for all $x\in\mathbb{R}^d$ and $i\in \{1, \cdots, d\}$. Note that
for such a function, its values on the subset $\Omega:=[0,1]^d$ of $\mathbb{R}^d$ are sufficient
to define its entire values on the whole space $\mathbb{R}^d$. Based on this remark, we will
identify in the sequel 1-periodic functions on $\mathbb{R}^d$ with their restrictions on
$\Omega=[0,1]^d$. Now, let $\mu$ and $\nu$ be two probability measures absolutely
continuous with respect to the Lebesgue measure on $\mathbb{R}^d$, and assume that their
respective densities $f$ and $g$ are 1-periodic.  Then $f, g: \Omega \rightarrow
\mathbb{R}$, and the $L^2$ optimal transport problem with these densities reads as
\begin{equation}\label{L2Mongeproblem}
\inf_{T:\Omega\rightarrow \Omega}\bigg\{ \int_\Omega |x-T(x)|^2\,\mbox{d}\mu(x); \;T_{\#}\mu=\nu,\, \mbox{d}\mu(x)=f(x)\mbox{d}x, \,\mbox{d}\nu(x)=g(x)\mbox{d}x\bigg\}
\end{equation}
Moreover,  the unique solution $\widetilde{T}=\nabla\Psi$ to this problem (where $\Psi:
\Omega\rightarrow \Omega$ is convex) satisfies the Monge-Amp\`ere equation
(\ref{MongeAmpere}) on $\Omega$. The regularity and boundary conditions
corresponding to this Monge-Amp\`ere equation are given by the following theorem due to 
Cordero--Erausquin \cite{CE}.

\begin{thm}\label{CE1}
 Assume that $\mu, \nu, f$ and $g$ are defined as above and let $m_{f}, m_{g},$ $M_{f}$
and $M_{g}$ denote the infima and suprema of $f$ and $g$ respectively. Then there exists a convex
function $\Psi: \Omega\rightarrow \Omega$ which pushes $\mu$ forward to $\nu$ (i.e.
$(\nabla \Psi)\#\mu=\nu$) such that $\nabla\Psi$ is additive in the sense that
$\nabla \Psi(x+p)=\nabla \Psi(x)+p$ for almost every $x\in\mathbb{R}^{d}$ and for
all $p\in\mathbb{Z}^{d}$. Moreover, $\nabla\Psi$ is unique and invertible ($\mu$
a.e.), and its inverse $\nabla\Phi$ satisfies $(\nabla \Phi)\#\nu=\mu$.  In
addition if $f$ and $g$ are of class $\mathcal{C}^{\alpha}(\Omega)$ with $\alpha>0$
and if $m_{g},M_{f}>0$, then $\Psi\in\mathcal{C}^{2,\beta}(\Omega)$ for some
$0<\beta<\alpha$ and it is a convex solution of the Monge-Amp\`{e}re equation
(\ref{MongeAmpere}).
 
\end{thm}

Note that since $\nabla\Psi$ is additive, it can be written as $x$ plus the gradient 
of a 1-periodic function. Thus, we assume $\Psi(x)=|x|^2/2 + u(x)$ with $\nabla u(x+p)=\nabla u(x)$
for all $p\in \mathbb{Z}^{d}$, i.e. $u$ is 1-periodic. So by using this change of function, $\Psi(x)=|x|^2/2 +u(x)$, in  the Monge-Amp\`ere equation (\ref{MongeAmpere}), we see that the
corresponding equation in $u$ satisfies a periodic boundary condition on $\Omega$.
This justifies why we introduce this change of function in Section \ref{sec_algo} to
rewrite Eq. (\ref{MongeAmpere}). In fact, the periodic boundary conditions will
allow us to use interior a priori estimates for classical solutions of the
Monge-Amp\`ere equation on the whole domain in order to prove the convergence of our algorithm (see
Section \ref{sect_proof}). We also infer from this theorem that if $f, g \in
C^\alpha(\Omega)$, then $\Psi \in C^{2,\beta}(\Omega)$ is the unique (up to a
constant) convex  solution of  the Monge-Amp\`ere equation (\ref{MongeAmpere}) on
$\Omega$. Finally, classical bootstraping arguments from the theory of elliptic regularity
can be used to prove that if $f,g\in\mathcal{C}^{k,\alpha}(\Omega)$, then $\Psi\in\mathcal{C}^{k+2,\beta}(\Omega)$.

\section{The Damped Newton algorithm}\label{sec_algo} 

\subsection{Derivation of the algorithm}\label{subsec_algo1}

Loeper and Rapetti  presented in \cite{LR} a numerical method based on Newton's
algorithm to solve the equation
$$
 \det(D^{2}\Psi)=f(x)
$$
in a periodic setting. This equation can be associated with the optimal transport problem in
the case where the target measure $\nu$ has a uniform density, i.e. $g=1$. Here, we propose to
extend this algorithm and the underlying analysis to the general case of an arbitrary smooth
1-periodic density $g$. Motivated by the remark made in Section \ref{sec_transp}, 
we follow \cite{LR} and introduce the change of function
$\Psi(x)=| x |^{2}/2+u(x)$ to rewrite the Monge-Amp\`{e}re equation
(\ref{MongeAmpere}) in the equivalent form
\begin{equation}\label{ma2}
 M(u)=g(x+\nabla u(x))\det(\mathcal{I}+D^{2}u(x))=f(x).
\end{equation}
Therefore, we will solve (\ref{ma2}) for a 1-periodic solution $u$ such that
$|x|^2/2 + u$ is convex on $\Omega=[0,1]^{d}$. Since we want to develop an algorithm
based on Newton's method, we first linearize (\ref{ma2}). Indeed, using the
formula for the derivative of the determinant \cite{LR,PP}, we have
$$
 \det{(\mathcal{I}+D^{2}(u+s\theta))}=
  \det{(\mathcal{I}+D^{2}u)}+s\,{\rm Tr}\,({\rm
Adj}(\mathcal{I}+D^{2}u)D^{2}\theta)+\mathcal{O}(s^{2})
$$
where ${\rm Adj}(A)=\det{(A)}\cdot A^{-1}$. Also, from the usual Taylor expansion,
we have
$$
 g(x+\nabla(u+s\theta))=g(x+\nabla u)+s\,\nabla g(x+\nabla u)\cdot \nabla
\theta+\mathcal{O}(s^{2}).
$$
Multiplying the latter two expressions, we obtain that the derivative in direction $\theta$ of the right hand side of equation (\ref{ma2}), denoted $D_{u}M\cdot\theta$, is given by
\begin{equation}
 g(x+\nabla u)\, {\rm Tr}\,({\rm Adj}(\mathcal{I}+D^{2}u)D^{2}\theta)+\det{(\mathcal{I}+D^{2}u)}\,\nabla g(x+\nabla u)\cdot \nabla
\theta\label{linearmaeq}.
\end{equation}
With this linearization at hand, we can now present the damped Newton algorithm that
we will use to solve equation (\ref{ma2}).

\smallskip
 \textbf{Damped Newton algorithm}
\begin{equation}
\left\{        \begin{array}{ll}
        {\rm With }\,u_{0}\, {\rm given,\, loop\, over\, }n\in\mathbb{N}\\
        \rule{0pt}{4ex}
         {\rm Compute\, }f_{n}=g(x+\nabla u_{n})\,\det{(\mathcal{I}+D^2{u_{n}})} \\
        \rule{0pt}{4ex}
         {\rm Solve\, the\, linearized\, Monge\mbox{-} Amp\grave{e}re\,\, equation} \\
        \rule{0pt}{4ex}
        \qquad\qquad\displaystyle D_{u_{n}}M\cdot\theta_{n} = \frac{1}{\tau}(f-f_{n}) \label{linma}\\
        \rule{0pt}{3ex}
        {\rm Update\, the\, solution:\, } u_{n+1}=u_{n}+\theta_{n}
        \end{array}
\right.        
\end{equation}
The factor $1/\tau$ ($\tau\geq 1$) in the algorithm  is used as a step-size parameter
to help preventing the method from diverging by taking a step that goes ``too far''.
As we will see below, the value of $\tau$ is crucial for the proof of convergence of the 
algorithm. Indeed, we will show that if we start with a constant initial guess for the Newton method, 
then there is a $\tau$ such that the method will converge (provided some extra conditions on the densities
are satisfied). Furthermore, by modifying some results presented in \cite{GT}, it is possible to prove that a
second order linear strictly elliptic pde with periodic boundary conditions has a
unique solution up to a constant if its zeroth-order coefficient is $0$. The linearized
Monge-Amp\`{e}re equation at step $n$, which we will
denote by $L_{n}$, falls into this category through the setting of the algorithm. To
fix the value of that constant, we select the solution satisfying
$\int_{\Omega}u\,dx=0$. This is guaranteed by choosing a $\theta_{n}$ which satisfies this 
condition for every $n$.

\subsection{Proof of Convergence}\label{sect_proof}

To prove the theoretical convergence of our algorithm (\ref{linma}), we follow the arguments in \cite{LR}, but we introduce several new key steps to deal with the non-uniform final 
density $g$. In particular, we rely on three
a priori estimates for the solution of the Monge-Amp\`{e}re equation. The first one
is derived by Liu, Trudinger and Wang \cite{LTW} and goes as follows: if 
 $\lambda\leq f/g \leq \Lambda$ for some positive constants $\lambda,\Lambda$, and if
$f\in C^{\alpha}(\Omega)$ for some $\alpha\in(0,1)$, then a convex solution of
(\ref{MongeAmpere}) satisfies 
 \begin{equation}\label{E1}
  \Vert \Psi \Vert_{\mathcal{C}^{2,\alpha}(\Omega_{r_{1}})}\leq
R_{1}\left[1+\frac{1}{\alpha(1-\alpha)}\bigg\Vert
\frac{f(x)}{g(\nabla\Psi(x)))}\bigg\Vert_{\mathcal{C}^{\alpha}(\Omega)}\right]
 \end{equation}
where $\Omega_{r_{1}}=\{x\in\Omega:{\rm dist}(x,\partial\Omega)>r_{1}\}$ and $R_{1}$
is a constant which depends only on $d$, $r_{1}$, $\lambda$, $\Lambda$ and $\Omega$.
The second one, discovered by Caffarelli \cite{V} and expressed by Forzani and
Maldonado \cite{FM},  presents a bound on the H\"{o}lder component of $\nabla\Psi$.
It states that if $f$ and $g$ are as in the previous estimate, there exists some constant $k$
such that
\begin{equation}\label{E2}
 { |\nabla\Psi(z)-\nabla\Psi(y) | \over |z-y|^{{1 \over 1+k}}}\leq R_{y}\left(K\over
m(\psi^{*}_{y},0,R_{y}) \right)^{1 \over 1+k}\left(M(\Psi,y,r_{2})\over r_{2}
\right)^{1\over 1+k}
\end{equation}
for $|z-y|\leq r_{2}$, where $\psi_{y}(x)=\Psi(x+y)-\Psi(y)-\nabla\Psi(y)\cdot x$,
$$
\displaystyle R_{y}=\max\bigg\{1,\left(KM(\Psi,y,r_{2})\over m(\psi^{*}_{y},0,1)
\right)^{k\over 1+k} \bigg\}
$$
and $M(\Psi,y,r_{2}),\,m(\Psi,y,r_{2})$ denote respectively the maximum and minimum of
$\psi_{y}(z-y)$ taken over the points $z$ such that $|z-y|=r_{2}$. Since this
estimate does not hold for all $k$, one might wonder for which values it is actually
valid. In \cite{FM},  it is shown that it holds for
$$
 k=2K(K+1),\quad
K=\frac{2^{3d+2}\,w_{d}\,w_{d-1}\,\Lambda}{d^{-3/2}\,\lambda}\quad
{\rm and }\quad
 w_{m}=\frac{\pi^{m/2}}{\Gamma(m/2+1)}.
$$
Here $w_{m}$ is the volume of the $m$-dimensional unit ball. To give the
reader an idea of these values, a few of them are presented in Table \ref{table2} below.
\renewcommand{\arraystretch}{1.5}
\begin{table}[ht]
\begin{center}
\begin{tabular}{|c|c|c|}
\cline{1-3}
d & $w_{d}$ & Rounded $K$ \\
\cline{1-3}
1 & 2 & $\displaystyle64C$ \\
\cline{1-3}
2 & $\pi$ & $\displaystyle4550C$ \\
\cline{1-3}
3 & $\displaystyle4\pi/3$ & $140039C$ \\
\cline{1-3}
4 & $\displaystyle\pi^{2}/2$ & $2709370C$ \\
\cline{1-3}
\end{tabular}
\caption{Quantities involved in the bounds on $\frac{1}{1+k}$ for $\displaystyle
C=\frac{\Lambda}{\lambda}$}\label{table2} 
\end{center}
\end{table}
\renewcommand{\arraystretch}{1}

\noindent Finally, the third estimate controls the growth of the second derivatives
of $\Psi$ with respect to boundary values, provided $f/g\in C^{2}(\Omega)$ and
$\Psi\in C^{4}(\Omega)$ \cite{GT,TW}:
\begin{equation}\label{E3}
\sup_{\Omega}|D^{2}\Psi|\leq R_{3}\left(1+\sup_{\partial \Omega}|D^{2}\Psi|\right)
\end{equation}
where $R_{3}$ depends only on $\Omega,d,f/g$ and on the $\sup$ of $\Psi+\nabla\Psi$
in $\Omega$. We can now state and prove the theorem on the convergence of
Algorithm (\ref{linma}). Note that the arguments of the proof are similar to \cite{LR}, but the fact that
the target density $g$ is non-uniform here introduces some new difficulties that are worthwhile 
exposing. In addition, the proof provides important information that can be used to gain 
some intuition about the performance of the algorithm in practice. 

\begin{thm}\label{convproof}
 Assume that $\Omega=[0,1]^{d}$ and let $f,g$ be two positive 1-periodic probability
densities bounded away from 0. Assume that the initial guess $u_{0}$ for the Newton
algorithm (\ref{linma}) is constant. Then if $f\in
\mathcal{C}^{2,\alpha}(\Omega)$ and $g\in
\mathcal{C}^{3,\alpha}(\Omega)$  for any $0<\alpha<1$, then there exits $\tau\geq1$ such
that $(u_{n})$ converges in $\mathcal{C}^{4,\beta}(\Omega)$, for any
$0<\beta<\alpha$, to the unique -- up to a constant -- solution $u$ of the
Monge-Amp\`{e}re equation (\ref{ma2}). Moreover, $\tau$ depends only on $\alpha, d$, $\Vert f
\Vert_{\mathcal{C}^{2,\alpha}(\Omega)}, \Vert g
\Vert_{\mathcal{C}^{3,\alpha}(\Omega)}$ and $M_{f}, M_{g}, m_{f},m_{g}$ which are defined as
in theorem \ref{CE1}.
\end{thm}
\begin{proof}
First we note that due to the additivity of the transport map, by applying
the change of variable $y=\widetilde{T}_{n}(x)=x+\nabla u_{n}(x)$, we can prove that for all
$n$,
$$
 \int_{\widetilde{T}_{n}(\Omega)}g(y)\,dy=\int_{\Omega}g(\widetilde{T}_{n}(x))\det(D\widetilde{T}_{n}(x))\,dx=\int_{\Omega}f_{n}\,dx=1,
$$
i.e. at every step, we are solving the optimal transport
problem sending $f_{n}$ to $g$. Moreover, unless otherwise stated, we only need to assume that $f\in\mathcal{C}^{\alpha}(\Omega)$ and $g\in\mathcal{C}^{2,\alpha}(\Omega)$. The main steps
of the proof consist in showing by induction that the following claims hold for all $n$ :
\begin{enumerate}
  \item $\mathcal{I}+D^{2}u_{n}$ and ${\rm Adj}(\mathcal{I}+D^{2}u_{n})$ are
$\mathcal{C}^{\alpha}(\Omega)$ smooth, uniformly positive definite (u.p.d.) matrices, where $\mathcal{I}$ denotes the identity matrix.
 \item $\displaystyle\frac{f}{C_{1}}\leq f_{n}\leq C_{1}f$, where $C_{1}$ is independent of $n$.
 \item $\Vert f-f_{n}\Vert_{\mathcal{C}^{\alpha}(\Omega)}\leq C_{2}$, where $C_{2}$
is independent of  $n$.
\end{enumerate}
We say that a matrix $A$ is $\mathcal{C}^{\alpha}(\Omega)$ smooth if all of its
coefficients are in $\mathcal{C}^{\alpha}(\Omega)$; it is u.p.d. (uniformly positive definite) if there exists
 a constant $k>0$ such that $\xi^{T}A\,\xi\geq k|
\xi|^{2}$ for all $\xi\in\Omega$. It is also worth mentioning that the
statement in 1) actually implies that $\Psi_{n}=|x|^{2}/2+u_{n}$ is uniformly convex and
that $L_{n}$ is a strictly elliptic linear operator. 

Note that for $u_{0}$ constant, we have
$f_{0}=g$. Next, let
$$
 C_{1}=\max\left(\frac{M_{f}}{m_{g}},\frac{M_{g}}{m_{f}}\right)\quad {\rm and}\quad
C_{2}=\Vert f \Vert_{\mathcal{C}^{\alpha}(\Omega)} +  \Vert
g\Vert_{\mathcal{C}^{\alpha}(\Omega)}.  
$$
Then, it is easy to see that all the claims 1), 2) and 3) hold for $n=0$. Let's assume they hold for a certain $n\in\mathbb{N}$ and prove them for $n+1$.
 For now, we suppose that the step-size parameter 
could vary with $n$. We shall prove later that we can actually take it to be constant without
affecting any result. Let $\theta_{n}$ be the unique solution of $L_{n}\theta_{n}=(f-f_{n})/\tau$ such that
$\int_{\Omega}\theta_{n}\,dx=0$. According to the results of \cite{GT} (modified for the periodic case)
there exists a constant $k_{\theta_{n}}$ such that
\begin{equation}\label{boundtheta}
 \big\Vert \theta_{n} \big\Vert _{\mathcal{C}^{i,\alpha}({\Omega})}\leq
\frac{k_{\theta_{n}}}{\tau}\big\Vert
f-f_{n}\big\Vert_{\mathcal{C}^{\alpha}({\Omega})}\leq
\frac{k_{\theta_{n}}C_{2}}{\tau},\quad i=1,2.
\end{equation}
Because $u_{n+1}=u_{n}+\theta_{n}$, we deduce that $\mathcal{I}+D^{2}u_{n+1}$ and
then Adj$(\mathcal{I}+D^{2}u_{n+1})$ are $\mathcal{C}^{\alpha}(\Omega)$ smooth.
Now, since $\mathcal{I}+D^{2}u_{n}$ is u.p.d., by assumption we get:
\begin{eqnarray}
 \displaystyle\xi^{T}(\mathcal{I}+D^{2}u_{n+1})\xi &=&
\xi^{T}(\mathcal{I}+D^{2}u_{n})\xi + \xi^{T}(D^{2}\theta_{n})\xi\nonumber\\
&\geq& K_{1} | \xi |^{2}-\frac{k_{\theta_{n}}}{\tau}\big\Vert
f-f_{n}\big\Vert_{\mathcal{C}^{\alpha}({\Omega})}\sum_{i,j=1}^{d}\xi_{i}\xi_{j}\nonumber\\
&\geq& K_{1} | \xi |^{2}-\frac{k_{\theta_{n}}}{2\tau}\big\Vert
f-f_{n}\big\Vert_{\mathcal{C}^{\alpha}({\Omega})}\sum_{i,j=1}^{d}\left(\xi_{i}^{2}+\xi_{j}^{2}\right)\nonumber\\
&=& \left[K_{1}-\frac{k_{\theta_{n}}d}{\tau}\big\Vert
f-f_{n}\big\Vert_{\mathcal{C}^{\alpha}({\Omega})}\right]| \xi |^{2}\nonumber\\
&\geq& K_{2}| \xi |^{2},\nonumber
\end{eqnarray}
for $\tau$ large enough, where $K_{2}$ is a positive constant. Hence $\mathcal{I}+D^{2}u_{n+1}$ is a u.p.d.
matrix. Next, inspired by the Taylor expansions previously shown, we write $f_{n+1}$
in terms of $f_{n}$ as follows:
\begin{eqnarray}
 f_{n+1}&=&g(x+\nabla u_{n+1})\det{(\mathcal{I}+D^{2}u_{n+1})}\nonumber\\
 &=& g(x+\nabla u_{n})\det{(\mathcal{I}+D^{2}u_{n})}+L_{n}\theta_{n}+r_{n}\nonumber\\
 &=&f_{n}+\frac{f-f_{n}}{\tau}+r_{n}\label{eq1}.
 \end{eqnarray}
Now we bound the residual $r_{n}$. It is easy to see that an explicit formula for
$r_{n}$ can be obtained from the second order terms of the Taylor expansion of the Monge-Amp\`{e}re
operator, and it consists of a sum of a bunch of products of at least
two first or second derivatives of $\theta_{n}$ with $g$ and its derivatives
evaluated at $\nabla \Psi_{n}$ and with second derivatives of $\Psi_{n}$. By (\ref{boundtheta}), we know that we can bound the
$\mathcal{C}^{\alpha}(\Omega)$ norm of the second derivatives of $\theta_{n}$ by a
constant times $\Vert f-f_{n}\Vert_{\mathcal{C}^{\alpha}(\Omega)}/\tau$. In
addition, since $g\in\mathcal{C}^{2,\alpha}(\Omega)$, the H\"{o}lder norm of $g$
and its first and second derivatives are all uniformly bounded. We then deduce that
\begin{equation}
\Vert r_{n} \Vert_{\mathcal{C}^{\alpha}(\Omega)}\leq \frac{k_{{\rm
r}_{n}}}{\tau^{2}}\Vert
f-f_{n}\Vert_{\mathcal{C}^{\alpha}(\Omega)}^{2},\label{eq2}
\end{equation}
where $k_{{\rm r}_{n}}$ could potentially depend on the H\"{o}lder norms of the
first and second derivatives of $\Psi_{n}$. Next, by selecting $\tau\geq
2k_{{\rm r}_n}C_{2}$, ($\ref{eq1}$), ($\ref{eq2}$) and 3) imply
\begin{eqnarray}
  \displaystyle \Vert f-f_{n+1} \Vert_{\mathcal{C}^{\alpha}(\Omega)}&\leq&
\left(1-\frac{1}{\tau} \right)\Vert
f-f_{n}\Vert_{\mathcal{C}^{\alpha}(\Omega)}+\frac{k_{{\rm
r}_{n}}}{\tau^{2}}\Vert
f-f_{n}\Vert_{\mathcal{C}^{\alpha}(\Omega)}^{2}\nonumber\\
  &\leq& \Vert
f-f_{n}\Vert_{\mathcal{C}^{\alpha}(\Omega)}\left(1-\frac{1}{\tau}+\frac{k_{{\rm
r}_{n}}C_{2}}{\tau^{2}}\right)\nonumber\\
  &=&\left(1-\frac{1}{2\tau}\right)\Vert
f-f_{n}\Vert_{\mathcal{C}^{\alpha}(\Omega)}.\nonumber
\end{eqnarray}
This shows that bound 3) is preserved for $\tau$ large enough. In addition,
it shows that we can take the step-size $\tau$ such that the sequence of bounds
$K_{2}$ created recursively will converge to a constant strictly greater than $0$.
Let's now verify bound 2). If we take $\tau\geq k_{{\rm
r}_{n}}C_{2}^{2}/[m_{f}(1-\frac{1}{C_{1}})]$, from all the previous results and
hypothesis we get 
\begin{eqnarray}
 \displaystyle f-f_{n+1}&\leq& \frac{\tau-1}{\tau}(f-f_{n})+\frac{k_{{\rm
r}_{n}}}{\tau^{2}}\Vert f-f_{n}
\Vert_{\mathcal{C}^{\alpha}(\Omega)}^{2}\nonumber\\
 &\leq& \frac{\tau-1}{\tau}f\left(1-\frac{1}{C_{1}}\right)+\frac{k_{{\rm
r}_{n}}C_{2}^{2}}{\tau^{2}}.\nonumber\\
 &=&f\left(1-\frac{1}{C_{1}}\right)\nonumber
\end{eqnarray}
from which we deduce that $f/C_{1}\leq f_{n+1}$. Following a similar approach with a
step-size $\tau\geq k_{{\rm r}_{n}}C_{2}^{2}/[m_{f}(C_{1}-1)]$, we obtain the
other part of 2). Then, we go back and finish the proof of the first statement.
Knowing that $\mathcal{I}+D^{2}u_{n+1}$ is u.p.d., we see that
$\det{(\mathcal{I}+D^2u_{n+1})}>0$ and therefore $\mathcal{I}+D^2u_{n+1}$ is
invertible. We can prove that its inverse is also a u.p.d. matrix. Indeed, if $\xi=(\mathcal{I}+D^{2}u_{n+1})y\in\Omega$, we have
\begin{eqnarray} 
\xi^{T}\left(\mathcal{I}+D^{2}u_{n+1} \right)^{-1}\xi &=& \bigg[\left((\mathcal{I}+D^{2}u_{n+1})y\right)^{T}\left(\mathcal{I}+D^{2}u_{n+1}
 \right)^{-1} \nonumber \\ && \qquad\qquad\qquad\qquad\qquad\left((\mathcal{I}+D^{2}u_{n+1})y\right)\bigg] \nonumber \\
=\quad y^{T}\left(\mathcal{I}+D^{2}u_{n+1} \right)y &\geq& K_{3}| y |^{2} \,\,\,=\,\,\, K_{3}| (\mathcal{I}+D^{2}u_{n+1})^{-1}\xi |^{2}.\nonumber
\end{eqnarray}
Using the inequality $| AB | \leq | A |\, | B |$ with $A=\mathcal{I}+D^{2}u_{n+1}$ and
$B=(\mathcal{I}+D^{2}u_{n+1})^{-1}\xi$, we obtain $| \xi | \leq |
\mathcal{I}+D^{2}u_{n+1} |\, | (\mathcal{I}+D^{2}u_{n+1})^{-1}\xi |$. Next,
motivated by the equivalence of norms, we use the bounds we derived previously to
get
\begin{eqnarray*}
 | \mathcal{I}+D^{2}u_{n+1} |&\leq& d\max_{i,\,j}\bigg\{|
(\mathcal{I}+D^{2}u_{n+1})_{ij}| \bigg\}\\
 &\leq& d\bigg(1+\Vert u_{n}\Vert_{\mathcal{C}^{2,\alpha}(\Omega)}+\Vert
\theta_{n}\Vert_{\mathcal{C}^{2,\alpha}(\Omega)}\bigg) \leq K_{4}
\end{eqnarray*}
where $K_{4}$ is a positive constant. This yields the claim:
$$
 \displaystyle \xi^{T}\left(\mathcal{I}+D^{2}u_{n+1} \right)^{-1}\xi \geq
\frac{K_{3}| \xi | ^{2}}{| \mathcal{I}+D^{2}u_{n+1} |^{2}}
 \geq\frac{K_{3}}{K_{4}^{2}}| \xi |^{2}
= K_{5}| \xi |^{2},\,\,\,  K_{5}>0.\nonumber
$$
We now use these statements to show that $L_{n+1}$ is a strictly elliptic operator,
\begin{eqnarray*}
g(x+\nabla u_{n+1})\sum_{i,j=1}^{d}{\rm
Adj}\left(\mathcal{I}+D^{2}u_{n+1}\right)_{ij}\xi_{i}\xi_{j}&\geq& f_{n+1}K_{3}\,| \xi
|^{2}\\
&\geq& \frac{f}{C_{1}}K_{3}\,| \xi |^{2}
 \geq K_{6}\,| \xi |^{2}.
\end{eqnarray*}
Note that by removing $g$ from the previous inequalities, we get that ${\rm
Adj}(\mathcal{I}+D^{2}u_{n+1})$ is a u.p.d. matrix, which completes the proof of 1). Now, we show that the step-size $\tau$ 
can be taken constant, as claimed before.
Indeed, 1) gives $\Psi_{n}\in\mathcal{C}^{2,\alpha}(\Omega)$ by
construction while 2) yields $f_{n}\in\mathcal{C}^{\alpha}(\Omega)$ and
$$
 0<\frac{m_{f}}{C_{1}M_{g}}\leq \frac{f_{n}(x)}{g(x+\nabla u_{n})} \leq
\frac{C_{1}M_{f}}{m_{g}}.
$$
Therefore, all the conditions to the estimate (\ref{E1}) are satisfied at every
step. Using inequalities on H\"{o}lder norms, we find
\begin{equation}
 \bigg\Vert \frac{f_{n}}{g(\nabla \Psi_{n})} \bigg\Vert_{\mathcal{C}^{\alpha}(\Omega)}
 \leq \Vert f_{n} \Vert_{\mathcal{C}^{\alpha}(\Omega)} \bigg\Vert \frac{1}{g}
\bigg\Vert_{\mathcal{C}^{\sqrt\alpha}(\Omega)}\left(1+\Vert \nabla \Psi_{n}
\Vert^{\sqrt\alpha}_{\mathcal{C}^{\sqrt\alpha}(\Omega)}\right)\nonumber
\end{equation}
At this point, the only remaining challenge is to bound $\Vert \nabla \Psi_{n}
\Vert_{\mathcal{C}^{\sqrt\alpha}(\Omega)}$. It can be achieved through the second
estimate (\ref{E2}). Since $\nabla\Psi_{n}$ is the transport map moving
$f_{n}$ to $g$, we can refer to Theorem \ref{CE1} to deduce that $\nabla \Psi_{n}$
is invertible and thus $\nabla \Psi_{n}\in\Omega$, which in turn yields
$\Psi_{n}\in[0,\sqrt{d}]$ when $\Omega=[0,1]^{d}$. Therefore, we see that the
maximum terms $M(\Psi,y,r_{2})$ are going to be uniformly bounded and that the only
problem could come from the minimum terms $m(\psi_{n_{y}}^{*},0,a)$, $a=1$ or
$R_{y}$. Using ideas from convex analysis (presented for example in \cite{HL})), we can 
show that since $\Psi_{n}$ is uniformly convex for every $n$, we have
$m(\psi_{n_{y}}^{*},0,a)=\min\psi_{n_{y}}^{*}(z)$ where the minimum is taken on the
sphere $|z|=a$, $a\geq 1$ (with the periodicity we can increase the size of $\Omega$ to include
it inside and still have a uniform bound on $\Psi_{n}$ and $\nabla\Psi_{n}$).
Furthermore, $\nabla\psi_{n_{y}}^{*}(z)=0$ if and only if $z=0$, 
$\nabla\psi_{n_{y}}^{*}$ is strictly monotone increasing because
$\nabla\psi_{n_{y}}$ is and $\nabla\psi_{n_{y}}^{-1}=\nabla\psi_{n_{y}}^{*}$. We see
that the only possible breakdown happens when $\nabla\psi_{n_{y}}^{*}$ converges to
a function which is zero up to $|z|=a$. This means
$|\nabla\psi_{n_{y}}|=|\nabla\Psi_{n}(x+y)-\nabla\Psi_{n}(y)|\rightarrow\infty$ as
$|x|\rightarrow0$ and $n\rightarrow\infty$, for any $y$. Observe now that if we increase the regularity of the densities to $f\in\mathcal{C}^{2,\alpha}(\Omega)$, $g\in\mathcal{C}^{3,\alpha}(\Omega)$, we get $f_{n}\in\mathcal{C}^{2,\alpha}(\Omega)$ at every step. This tells us that $\theta_{n}\in\mathcal{C}^{4,\alpha}(\Omega)$ (see \cite{GT}) and thus $\Psi_{n}\in\mathcal{C}^{4,\alpha}(\Omega)$. Therefore, we can apply estimate (\ref{E3}) and rule out this potential breakdown case. We obtain that the $\mathcal{C}^{2,\alpha}(\Omega_{r})$ norm of  $\Psi_{n}$ is
uniformly bounded and thus by the additivity of that function in a periodic setting,
the same conclusion holds for its $\mathcal{C}^{2,\alpha}(\Omega)$ norm. Hence, we
deduce that it is also the case for $k_{r_{n}}$ and then $k_{\theta_{n}}$. 
From this, we get that we can select a $\tau\geq 1$ constant such that the three statements 
hold for all $n\in\mathbb{N}$ by induction. Moreover, the sequence $(u_{n})_{n\in\mathbb{N}}$ 
is uniformly bounded in $\mathcal{C}^{2,\alpha}(\Omega)$, thus equicontinuous. By the 
Ascoli-Arzela theorem, it converges uniformly in $\mathcal{C}^{2,\beta}(\Omega)$ for $0<\beta<\alpha$ to
the solution $u$ of (\ref{ma2}), which is unique since we impose
$\int_{\Omega}u\,dx=0$. Finally, due to the fact that the initial and final densities are actually
$\mathcal{C}^{2,\alpha}(\Omega)$, we know that this solution will be in
$\mathcal{C}^{4,\beta}(\Omega)$.
\end{proof}

\subsection{Remarks on the Proof}

This proof by induction provides a lot of precious information concerning the
properties of the iterates created by our method. First, since
$\mathcal{I}+D^{2}u_{n}$ is u.p.d. at every step, we realize that the sequence of functions
$\Psi_{n}$ is actually one of uniformly convex functions. Recall that the
Monge-Amp\`{e}re equation (\ref{MongeAmpere}) is elliptic only when we restrict it to the
space of convex functions. Therefore, the algorithm is extra careful by approximating the 
convex solution of the Monge-Amp\`{e}re equation
by a sequence of uniformly convex functions. In addition, this guarantees that the
linearized equation is strictly elliptic and thus has a unique solution (once we
fix the constant). Furthermore, just like in \cite{LR}, we can obtain estimates on
the speed of convergence of the method. Indeed, assuming that $\tau\geq 2k_{{\rm
res}}C_{2}$, we got
$$
  \displaystyle \Vert f-f_{n+1}
\Vert_{\mathcal{C}^{\alpha}(\Omega)}\leq\left(1-\frac{1}{2\tau}\right)\Vert
f-f_{n}\Vert_{\mathcal{C}^{\alpha}(\Omega)}\nonumber
$$
which tells us that $(f_{n})$ converges to $f$ following a geometric convergence
with a rate of at least $\displaystyle 1-1/2\tau$. When it comes to the step-size
parameter $\tau$, it would be very useful to know a priori which value to select in
order to make the algorithm converge. Such an estimate is unfortunately hard to
acquire since some of the constants used through interior bounds are obtained via
rather indirect arguments. However, we observe from lower bounds on $\tau$ used in
the proof, i.e.
$$
 \tau\geq \frac{k_{{\rm res}}C_{2}^{2}}{m_{f}\left(1-\frac{1}{C_{1}}\right)},\quad
\tau\geq\frac{k_{{\rm res}}C_{2}^{2}}{(C_{1}-1)m_{f}} {\rm\quad or\quad }\tau\geq
2k_{{\rm res}}C_{2},
$$
that the minimum value required on the step-size parameter to achieve convergence
could potentially be large when $m_{f}$ is close to $0$ or when either $\Vert f
\Vert_{\mathcal{C}^{\alpha}(\Omega)}$ or $\Vert
g\Vert_{\mathcal{C}^{\alpha}(\Omega)}$ is large. Through the numerous numerical
experiments we conducted, we realized that $\tau$ seems to behave according to both
conditions. Therefore, knowing a priori that $f$ could get close to $0$, we can react 
accordingly by either increasing the value of the
step-size parameter or by modifying the representation of the densities (which is
possible in some applications). Finally, even if our proof only guarantees
convergence when the update $\theta_{n}$ is the solution of (\ref{linearmaeq}), in
practice we can get good results by replacing it by the solution of $g(x+\nabla u_{n})\, {\rm Tr}\,({\rm
Adj}(\mathcal{I}+D^{2}u_{n})D^{2}\theta_{n})=(f-f_{n})/\tau$, or sometimes by an even simpler equation.

\section{Numerical Discretization}\label{sec_disc}

We present here a two-dimensional implementation of the Newton algorithm
(\ref{linma}). We consider a uniform $N\times N$ grid with a space-step $h=1/N$ where
we identify $x_{i}=0$ with $x_{i}=1$ ($i=1,2$) by the periodicity. It is easy to
see that the most important step for the efficiency of the method is the
resolution of the linearized Monge-Amp\`{e}re equation. Indeed if we take $P$ to
be the number of points on the grid ($P=N^{2}$ in 2D), as every other step can be
done in $\mathcal{O}(P)$ operations, the computational complexity of the whole
method is dictated by the resolution of this linear pde. Therefore, we will
introduce below two methods for solving this equation. For the other
steps, we employ fourth-order accurate centered finite differences for the
discretization of the first and second derivatives of $u_{n}$. We thus improve considerably
the accuracy of the results compared to \cite{LR} where second order differences are used to approximate
these terms, but at the same time we do not decrease the efficiency of the whole algorithm whose complexity is dominated by
the resolution of the linear PDE. If we know $g$ explicitly, then we can compute the compositions $g(x+\nabla u_{n})$
and $\nabla g(x+\nabla u_{n})$ directly. However, it is not always the case, especially when
we deal with discrete data as in the examples of image processing in Section \ref{sec_exp}. In
such circumstances, we have to approximate them. To do so, a popular choice would be to use a linear interpolation but in practice, we find that using only a closest neighbour interpolation gives
good results in most scenarios. Another salient point is that even though in theory $f_{n}$ has a
total mass of 1 at every step, it is not necessarily the case in the
numerical experiments, due to discretization errors. However, we need the
right-hand side of the linearized Monge-Amp\`{e}re equation to integrate to 0 on the
whole domain. To deal with this, we introduce a normalization step right after
computing $f_{n}$ in the implemented algorithm, taking
\begin{equation}\label{translation}
 \tilde{f}_{n}=f_{n}-\frac{1}{N^2}\sum_{i,j=0}^{N-1}f_{n_{i,j}}+1
\end{equation}
instead of $f_{n}$ and thus translating it at every step.

\subsection{A Finite Differences Implementation}
We begin by presenting an implementation of the resolution of the linearized
Monge-Amp\`{e}re equation through finite differences. This choice is motivated by
the fact that it is the method chosen by Loper and Rapetti in \cite{LR} for their
corresponding algorithm. In this case, to reduce the complexity of the code, we only
use centered finite differences of second-order for the derivatives of $\theta_{n}$.
Since the linear pde has a unique solution only up to a constant, the linear system
$Ax=b$ corresponding to its discretization has one free parameter that we need to
fix to turn the matrix into an invertible one. A possible strategy to achieve
this is to create a new system $\hat{A}x=\hat{b}$ by adding the extra equation $\sum
x_{k}=0$, which corresponds to selecting $\theta_{n}$ such that
$\int_{\Omega}\theta_{n}\,dx=0$. Note that this new matrix has full rank, but it is
not square. Then, we take that extra line, add it to all the other lines of
$\hat{A}$ and then delete it to get a square system, $\tilde{A}x=b$. The next lemma
shows  conditions under which the resolution of $\tilde{A}x=b$ will produce a valid
answer to the system $Ax=b$. For the sake of notation, consider the new equation to
be stored in the first line of $\hat{A}$. 
\begin{lem}
 Let $\hat{A}$ and $\tilde{A}$ as defined above, i.e., $\hat{A}$ is a
$(P+1)\times P$ matrix with rank $P=N^2$, and there exist real numbers
$\alpha_1,\alpha_2,\dots,\alpha_{P+1}$ not all zero such that $\alpha_1 L_1 +
\alpha_2 L_2 + \dots + \alpha_{P+1} L_{P+1}=0$ where $L_i$ is the $ith$ line of
$\hat{A}$. If $\alpha_2 + \dots + \alpha_{P+1} \neq \alpha_1$, then $\tilde{A}$ has
rank $P$. 
\end{lem}

The proof is a straightforward use of matrix algebra and therefore not reported here
for brevity. This lemma does not hold if condition $\alpha_2 + \dots +
\alpha_{P+1} \neq \alpha_1$ is not satisfied. Take for example a matrix $\hat{A}$ such that
 its second line is equal to the negative of its first line and all
its other lines are linearly independent. Then $\tilde{A}$ has rank $P-1$. Nonetheless,
due to the structure of our problem, this is not going to happen. Unfortunately,
this strategy has the downside of somewhat destroying the sparsity of the matrix.
One way to avoid this would be to equivalently fix only the
value of $\theta$ at a given point and then use the same strategy. This would
preserve most of the sparsity of the matrix.

Next, to actually solve the system $\tilde{A}x=b$, we employ the Biconjugate
Gradient (BICG) iterative method. This choice can be justified by the fact that we
are dealing with a (potentially sparse) matrix which is not symmetric nor positive
definite; the BICG procedure being specifically designed to deal with these
conditions \cite{S}. One feature of this method is that, provided the method does not break
down before, the sequence of approximate solutions it produces is guaranteed to
converge to the solution of the linear system in a maximum of P steps,  which yields
a computational complexity of at worst $\mathcal{O}(P^{2})$. However, as we shall
see later, in practice it can be much smaller than that. In addition, even if the
BICG algorithm is commonly employed with a preconditionner, we did not find the need
to do so while conducting our numerical experiments.

\subsection{A Fourier Transform Implementation}\label{fourier_impl}

One should realize that the first implementation we employed to solve the linearized
Monge-Amp\`{e}re equation might not be the best method. Indeed, there
exist much cheaper ways to solve a linear second-order strictly elliptic equation
with such boundary conditions. The one we are going to explore here is due to Strain
\cite{St} and requires only $\mathcal{O}(P{\rm Log}P)$ operations
through the use of the FFT algorithm. It consists in rewriting the problem as the system
\begin{equation}\label{fouriersystem}
\left\{        \begin{array}{ll}
            L_{n}\overline{L}_{n}^{\,-1}\sigma(x)=h(x)\\
            \rule{0pt}{3ex}
            \quad\,\,\,\,\,\,\overline{L}_{n}\theta(x)=\sigma(x),
        \end{array}
\right.
\end{equation}
where $\overline{L}_{n}$ is the averaged $L_{n}$ in the sense that its coefficients
are the integral over $\Omega$ of the coefficients of $L_{n}$. We then expand
$\sigma$ in Fourier series by taking
$$
 \sigma(x)=\sum_{k\neq0\in\mathbb{Z}^{d}}\widehat{\sigma}(k)e^{2\pi i k\cdot x}\quad
 {\rm and   }\quad\widehat{\sigma}(k)=\int_{\Omega}\sigma(x)e^{-2\pi i k\cdot
x}\,dx ,
$$
where $ i$ representing $\sqrt{-1}$ and $\widehat{\sigma}(k)$ being the usual Fourier
coefficients. Using this expansion in the first part of (\ref{fouriersystem}) yields
the formula 
\begin{eqnarray}
 L\overline{L}^{\,-1}\sigma(x)&=&\sum_{i,j=1}^{d}a_{ij}(x)\sum_{k\neq0\in\mathbb{Z}^{d}}2\pi i
k_{i}2\pi i k_{j}\overline{\rho}(k)\widehat{\sigma}(k)e^{2\pi i k\cdot
x}\nonumber\\
 &&\qquad\qquad\,\,\,\,+\sum_{i=1}^{d}b_{i}(x)\sum_{k\neq0\in\mathbb{Z}^{d}}2\pi i
k_{i}\overline{\rho}(k)\widehat{\sigma}(k)e^{2\pi i k\cdot x}\nonumber\\
 &=&\sum_{i,j=1}^{d}a_{ij}(x)\alpha_{ij}(x)+\sum_{i=1}^{d}b_{i}(x)\beta_{i}(x)\nonumber
\end{eqnarray}
where
\begin{equation}
 \overline{\rho}(k)=\left\{        \begin{array}{ll}
            \displaystyle \frac{1}{ \sum_{i,j=1}^{d}\overline{a}_{ij}2\pi i
k_{i}2\pi i k_{j}+\sum_{i=1}^{d}\overline{b}_{i}2\pi i
k_{i}}\qquad{\rm\, if\, the\, sum\, is \,not\, 0}\\
            0\qquad\qquad\qquad\qquad\qquad\qquad\,\,\qquad\qquad\qquad\quad{\rm otherwise}.
        \end{array}
\right.
\end{equation}
For the discretized problem, knowing the value of $\sigma$, we can compute
$\widehat{\sigma}$ with one application of the FFT algorithm and then compute
$\alpha_{ij}$ and $\beta_{i}$ with $d(d+1)$ applications of the inverse FFT
algorithm to be able to get the value of $L\overline{L}^{\,-1}\sigma(x)$ in
$\mathcal{O}(P{\rm Log}P)$ operations. Therefore, we can use an iterative method to
solve the first equation of system (\ref{fouriersystem}) at a cost of
$\mathcal{O}(P{\rm Log}P)$ operations per iteration. As in Strain \cite{St}, we use the
Generalized Minimal Residual method, or GMRES. Just like BICG, it is an efficient
way of solving a linear system of equations where the matrix is non-symmetric and
non-positive definite \cite{S}. Moreover, GMRES does not use projections on the Krylov
subspace generated by the transposed matrix. This makes it easier to code for the
particular setting we are dealing with since we do not form $A$ directly; we
reference it instead through the result of its product with a given vector $\sigma$.
Strain observed that the number of GMRES
iterations required did not vary with $P$, which yields a global complexity of
$\mathcal{O}(P{\rm Log}P)$. Note that for better performances, we actually employ like
the author the restarted GMRES(m) method. After computing $\sigma(x)$, we need to solve $\overline{L}\theta(x)=\sigma(x)$.
This can be easily achieved since we already know the value of $\overline{L}^{-1}\sigma(x)$.
More specifically, we have
$$
 \theta(x)=\sum_{k\neq0\in\mathbb{Z}^{d}}\overline{\rho}(k)\widehat{\sigma}(k)e^{2\pi i
k\cdot x},
$$
i.e. it requires only one other application of the (inverse) FFT algorithm to obtain
$\theta$. On top of the efficiency of this method, observe that it has other
advantages. It is spectrally accurate, i.e. the error decreases faster than any power
of the grid size as the space-step size goes to $0$. We can also prove that the
convergence rate for the GMRES algorithm is independent of the grid size. For more
details, one should consult the original paper \cite{St}. In the actual
discretization of this method, we truncate the sums in the usual way by varying
$|k|$ from $-N/2$ to $N/2$. We compute the averages of the operator's coefficients
with Simpson's numerical integration formula. Finally, the discrete linear system
still has a solution unique only up to a constant and we can use the same strategy
as in the previous case to fix it.

\section{Numerical Tests}\label{sec_exp}

\subsection{A theoretical example}

\begin{figure}[ht!]
\centering
\subfigure[The $\Vert u-u_{n}\Vert_{l^{2}}$ error for the Fourier transform
implementation on a semilog
plot]{\includegraphics[width=2.49in]{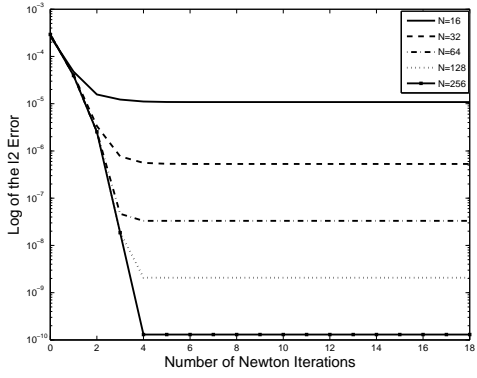}}
\,\,
\subfigure[The $\Vert u-u_{n}\Vert_{l^{2}}$ error for the finite differences
implementation on a semilog
plot]{\includegraphics[width=2.49in]{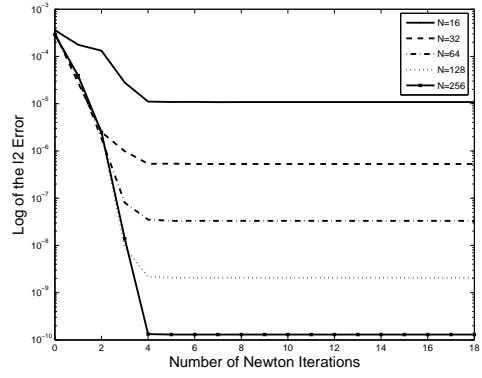}}\\

\subfigure[The $\Vert f-\tilde{f}_{n}\Vert_{l^{2}}$ error for the Fourier transform
implementation on a semilog
plot]{\includegraphics[width=2.49in]{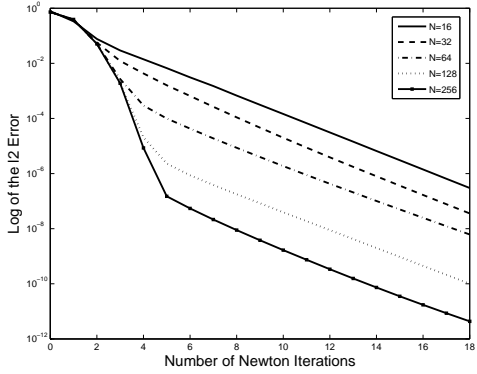}}
\,\,
\subfigure[The $\Vert f-\tilde{f}_{n}\Vert_{l^{2}}$ error for the finite differences
implementation on a semilog
plot]{\includegraphics[width=2.49in]{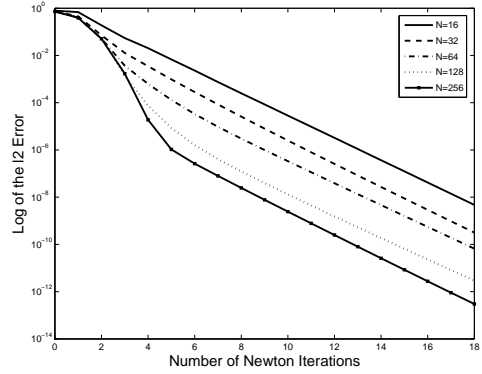}}
\caption{Error behavior for $u_{n}$ and $f_{n}$ for a tolerance of
$10^{-4}$.}\label{smalltol}
\end{figure}

\noindent Our goal here is to observe and compare the behaviour of the two
implementations. Starting with a known $u$ and a known $g$, we compute the
corresponding right-hand side $f$ with (\ref{ma2}), and then we run the algorithm to
obtain $u_{n}$. We consider functions of the form
\begin{eqnarray}
 \quad \displaystyle u(x_{1},x_{2})&=&\frac{1}{k}\cos\left(2\pi\gamma
x_{1}\right)\sin(2\pi\gamma x_{2}),\nonumber\\
 \quad \displaystyle g(x_{1},x_{2})&=&1+\alpha \cos\left(2\pi\rho
x_{1}\right)\cos(2\pi\rho x_{2})\nonumber.
 \end{eqnarray}
For the first implementation we select for the BICG algorithm a tolerance of $10^{-4}$ and a maximum
number of 1000 iterations per Newton step. For the FFT implementation, we take the same
tolerance with a restarting threshold of $m=10$ inner iterations for the GMRES algorithm.
In both cases, a value of $\tau=1$ was enough to achieve convergence. The errors $\Vert u-u_{n} \Vert_{l^{2}}$
and $\Vert f-f_{n} \Vert_{l^{2}}$ are plotted in Figure \ref{smalltol} as functions of the Newton
iterations for both the FFT and finite differences and for various grid sizes ranging from $16\times 16$
to $256\times 256$ grid points. We see that in both cases the error gets smaller as we increase the
grid size. In particular, for this value of $\tau$, after the first $4$ iterations or so, where 
$\Vert u-u_{n}\Vert_{l^{2}}$ settles down very quickly, the convergence of $\Vert f-f_{n}\Vert_{l^{2}}$
follows a linear slope with a convergence rate slightly faster
than a half. The estimated ratio is actually about $0.45$ in the FFT case and is about $0.33$ in the finite differences
case, so the convergence is faster in this latter case for this final stage. Computing the observed order of accuracy from the 
errors between $u$ and $u_{n}$, we get from smaller to bigger  
grid sizes,  $4.3521, 4.0035, 3.9965$ and $3.9990$. This confirms that the fourth-order is consistent with the order
of the finite difference scheme used to compute the right-hand side.

In order to investigate whether we can decrease the computing time without loosing too much
precision on the results, we try to run the experiment again with a tolerance
$10^{-1}$ (see Figure \ref{bigtol}). Due to the looser tolerance employed, the
results are a bit erratic for the finite differences implementation, but overall
still very good. Figure \ref{error_surface_plot} shows the 3d plot of $u-u_{n}$ for $N=128$ in the
Fourier transform case to get an idea of the distribution of the errors. As we can
see, they seem evenly distributed on the whole domain. Figure \ref{error_diff_tau}
depicts what happens when we vary the value of the step-size parameter $\tau$.
The results behave accordingly to our expectations, with a slower convergence
for a bigger $\tau$. Note that for this new tolerance, the computational cost of one
iteration is now much less expensive and the global computing time decreases a lot
in both cases. We can quantify this by looking at Table \ref{table_exp1}. Observe
that the BICG algorithm required less operations than the worst case scenario
$\mathcal{O}(P^{2})$. This being said, we still realize at first glance that the
FFT method is much faster than the finite difference method. The
number of GMRES iterations per Newton iteration stayed nearly constant as we increased the
grid size, which confirms the $\mathcal{O}(P{\rm Log}P)$ computational complexity.  
\begin{figure}[ht!]
\centering
\subfigure[The $\Vert u-u_{n}\Vert_{l^{2}}$ error for the Fourier transform
implementation on a semilog plot]{\includegraphics[width=2.49in]{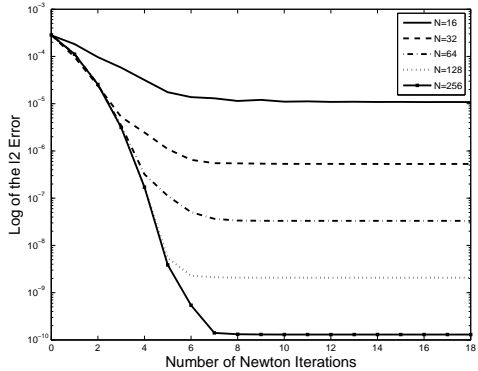}}
\,\,
\subfigure[The $\Vert u-u_{n}\Vert_{l^{2}}$ error for the finite differences
implementation on a semilog
plot]{\includegraphics[width=2.49in]{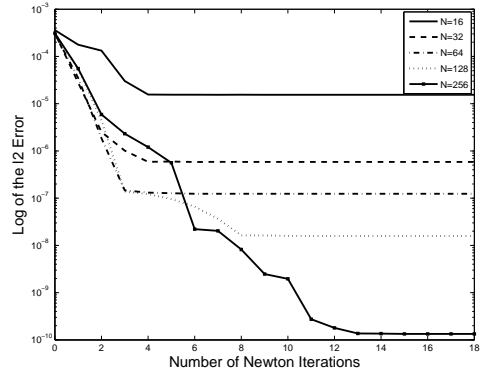}}

\subfigure[Surface plot of $u-u_{6}$ for the Fourier transform implementation in
the $N=128$ case.]{\includegraphics[width=2.49in]{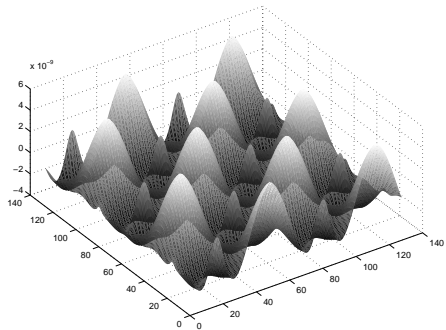}\label{error_surface_plot}}
\,\,
\subfigure[The $\Vert u-u_{n}\Vert_{l^{2}}$ error for the Fourier transform
implementation with different values of $\tau$ in the $N=128$ case on a semilog
plot]{\includegraphics[width=2.49in]{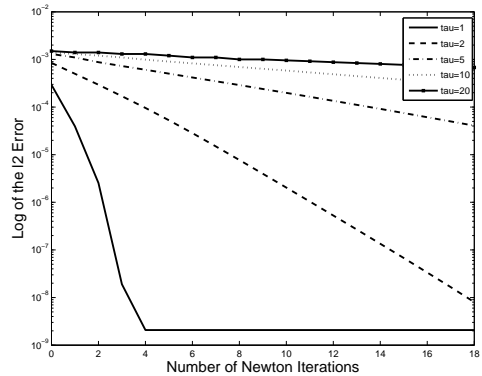}\label{error_diff_tau}}
\caption{Several examples of the results obtained with a tolerance of
$10^{-1}$.}\label{bigtol}
\end{figure}

\begin{table}
\begin{center}\footnotesize
\begin{tabular}{|c|c|c|c|c|}
\cline{1-5}
   &                     & Total           &                    & Total           \\
   & Average number      & computing time  & Average number     & computing time  \\
 N & of                  & for             & of                 & for             \\
   & GMRES iterations    & Fourier         & BICG iterations    & Finite          \\
   &                     & Transforms      &                    & Differences     \\
\cline{1-5}
16 & 5.32 & 1.07 & 14.21 & 2.21  \\
32 & 6.37 & 1.94 & 17.79 & 8.70 \\
64 & 7.32 & 8.06 & 31.11 & 79.17\\
128 & 7.95 & 34.38 & 63.32 & 1221.10 \\
256 & 8.05 & 145.07 & 134.63 & 34639.82\\
\cline{1-5}
\end{tabular}
\caption{Average number of BICG and GMRES iterations per Newton iteration and total
computing time in seconds for the whole experiment (20 iterations) when the
tolerance is set to $10^{-1}$. We used a MATLAB implementation on an Intel Xeon
running at 2.33 GHZ. This is presented for all the different grid sizes.}\label{table_exp1} 
\end{center}
\end{table}
Finally, in order to get an idea of the stability properties of both methods, we can measure
the norm of the inverse of the matrix corresponding to
the discretization of the linearized Monge-Amp\`{e}re operator (see for example \cite{L} for
more information on the stability concept for iterative methods). This can 
be achieved by computing the spectral radius of such matrix. In the finite
differences case, we could get this eigenvalue directly by first obtaining the
inverse, and then computing the eigenvalues of the new matrix. We observed that for
the current experiment, the spectral radius starts at about $10$ for $N=16$ and then
grows almost linearly as we increase the grid size. Hence, the finite difference
method appears unstable. For the FFT case, since we do not possess an
explicit representation of the matrix, we have to use an indirect method to compute
the spectral radius. The one we select is the power method (or power iteration).
For a matrix $A$, it starts with a vector $b_{0}$ and compute the iterates $b_{k+1}=Ab_{k}/\Vert Ab_{k} \Vert$. 
If $A$ has a dominant eigenvalue and if $b_{0}$ has a non-zero component in the direction of the eigenvector associated with this largest eigenvalue, 
then the sequence $(b_{k})$ converges to the eigenvector associated with the spectral radius of $A$ (see \cite{GL}). For the current experiment, we apply this technique to the inverse matrix produced at every Newton step $n$ by the discretization of the linearized 
Monge-Amp\`ere equation via the FFT implementation. More specifically, for a given $n$, we start with a $b_{0}$ randomly generated with components in $[0,1]$.
 Then, using the method presented in Section \ref{fourier_impl}, we compute the product $A_{n}^{-1}b_{k}$ and then the iterate with 
$A_{n}^{-1}b_{k}/\Vert A_{n}^{-1}b_{k}\Vert_{2}$. Repeating this procedure several times with different random vectors $b_{0}$, we observed that the power
 iteration always converges to a number close to $0.03$ after only about $k=5$ iterations, for every $n$ from $0$ to $20$ and for every $N$ from $16$ to $256$.
 Thus, we conclude for this specific example that the spectral radius of the inverse matrix generated at every step of the Newton method is close to $0.03$,
 which of course suggests that the FFT implementation is stable.

\subsection{Application to medical imaging}

Here, we test our algorithm on one of the various applications of optimal transport. In 
the area of image processing, one of the most common tasks performed by practitioners 
is to determine a geometric correspondence between images taken from the same scene in 
order to compare them or to integrate the information they contain to obtain more meaningful 
data. One could think of pictures acquired at different times, with different equipments or
from different viewpoints. This process falls into the category of what is referred
to as image registration. There are two main types of image registration methods:
the rigid ones which involve translations or rotations and the nonrigid ones where
some stretching of the image is required to map it to the other one.  People working on optimal
 transport recently realized that this theory could provide
a good nonrigid image registration technique. Indeed, consider for example two grayscale
images. We could think of them as representing a mass distribution of the amount of
light ``piled up" at a given location. A bright pixel on that image would then
represent a region with more mass whereas a darker pixel would correspond to a
region with less mass. Computing the optimal map between the two images and analyzing
the rate of change of that map could reveal the best way (in terms of minimizing the
transportation distance) of moving the mass from the first density to the second,
precisely showing what is changing on the images and how it is happening. In
\cite{R}, Rehman et al. actually lists several advantages of the optimal
transport method for image registration. However they also stress the fact that it
is computationally expensive and this is one reason why it is important to find
efficient numerical methods to solve this problem.

Our first applied test is in the field of medical imagery. We consider the two 
brain MRI scans presented in Figure \ref{exp2}. These images were taken from the 
BrainWeb simulated brain database at McGill university \cite{Br} and represent a slice of a 
healthy brain and a slice of the same brain where the multiple sclerosis disease is spreading.
This nervous system disease damages the myelin sheets around
the axons of the brain and leaves scars (or sclerosis) visible on an MRI. We chose
MS as a test case since its actual detection process relies on neuro-imaging which tries to 
identify the scars whose presence leaves traces similar to multiple tumours. 
Note that because the scans are dark, their representation in greyscale contains
many values close or equal to 0. For the obvious reason of accuracy, we normalize the densities
and bound them away from 0 by applying the translation in (\ref{translation}) 
on both of them before initiating the algorithm. In addition, we
rescale them so that they are exactly square ($256\times256$ pixels). This is
not required, but helps simplifying the code.
\begin{figure}[ht!]
\centering
\subfigure[Initial density $f$: MRI scan of a normal
brain]{\includegraphics[width=2.32in]{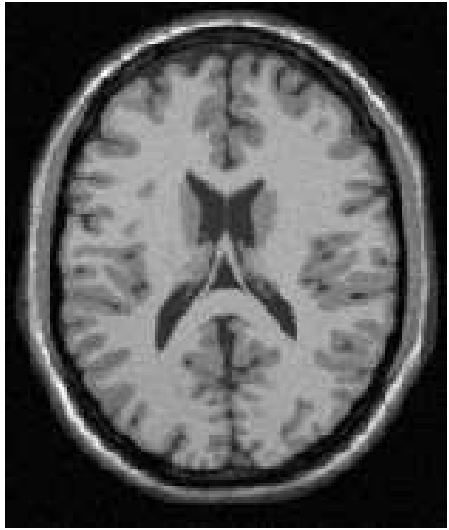}}
\,\quad\quad
\subfigure[Final density $g$: MRI scan of the same brain with Multiple Sclerosis
lesions]{\includegraphics[width=2.32in]{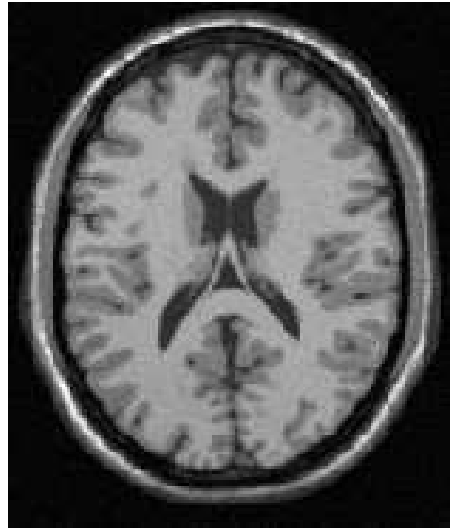}}

\subfigure[Surface plot of ${\rm
div}(u_{4})$]{\raisebox{8ex}{\includegraphics[width=2.59in]{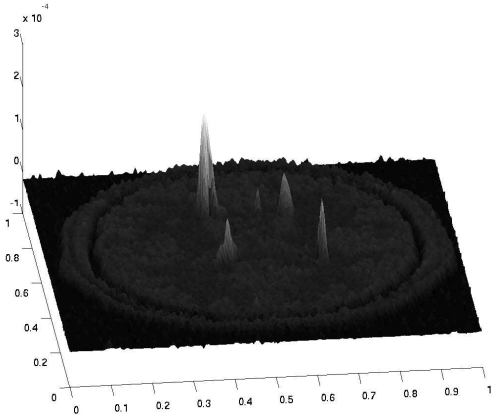}\label{brain_surf}}}
\,\,
\subfigure[Scan of the healthy brain on which was superposed the coloured filtered
contour plot of ${\rm
div}(u_{4})$]{\includegraphics[width=2.32in]{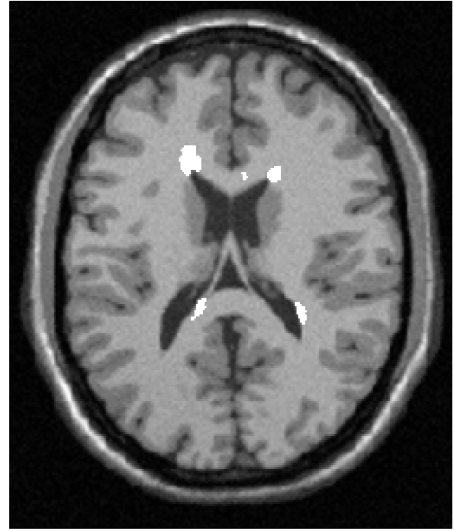}\label{brain_super}}
\caption{The results of the MS detection experiment with the Fourier transform
implementation}\label{exp2}
\end{figure}
On the bottom panels of Figure \ref{exp2}, we show the results reached after only 4 Newton
iterations with a $\tau=1$ and a tolerance of $10^{-2}$ for the Fourier transform
implementation. We note that the corresponding $L^{2}$ norm of the error between $f$ and $f_{n}$ is
reduced to about $0.001$. The 3d plot in Figure \ref{brain_surf} is characterized by very sharp spikes
corresponding to variations in brightness between the two images where the scars are located. To get a better visual understanding of the situation, we also took the graph of the filtered contour plot,
coloured in white the inside of the contour lines corresponding to the affected
regions and we superposed this image to the MRI scan of the healthy brain (Figure \ref{brain_super}). 
The number of GMRES iterations required per Newton iteration was very small and nearly constant
(only 1 outer iteration and about 6 inner ones). Even if our code was not
necessarily optimized in terms of speed, it only took about 30 seconds to compute
these results on an Intel Xeon with 2.33 GHZ of RAM. Moreover, the spectral radius
was still very close to $0.03$ for the inverse discretization matrix, which is very
encouraging. We also need to mention that we don't have access here to an analytical
expression for $f$ or $g$. Therefore, to compute an approximation for $g(x+\nabla
u_{n}(x))$ and $\nabla g(x+\nabla u_{n}(x))$ at every grid point $x=(ih,jh)$, we employ
a closest neighbour approximation and like previously pointed out, the results are
still very good. In addition to that change detection, the optimal transport plan $\widetilde{T}=x+\nabla
u(x)$ actually gives us precise information on the amount of variation from one MRI
scan to the other. Indeed, we can define a metric between probability densities from
the solution to the transport problem; the distance being
$$
 {\rm d}(f,g)=\int_{\Omega}| x-\widetilde{T}(x)|^{2}f(x)\,dx=\int_{\Omega}| \nabla
u(x)|^{2}f(x)\,dx.
$$
This could quantify the magnitude of the change between the two images and thus help
monitor the growth of the disease. In our experiment, we got a value of
$4.14\times10^{-10}$. Note that when we compared it to other ones obtained from
different numerical experiments on MRI scans with presence of MS, these numbers
validated our visual estimates; more scars would produce a bigger number. Recall
finally that even though we implemented it only in 2D, in theory it is valid in any
dimension. Therefore, it could also be applicable on 3D datasets which would be much
more realistic when it comes to analyzing biological phenomena similar to the
ones we treated here.

\subsection{A non-periodic example}  

So far, we selected examples that were well suited for our periodic boundary
conditions. Nonetheless, we claim that our algorithm could produce great performances
on densities which are not necessarily periodic. To demonstrate this, we chose two
famous pictures in the image processing community as initial and final densities, namely Lena and Tiffany. First,
we do a little bit of preprocessing by translating the initial pictures and by
scaling them to exactly $256\times256$ pixels. We point out that we did not apply
any smoothing to any of the images. Then, we select $\tau=2$, tol$=10^{-1}$ and run
20 iterations of the Newton algorithm with the Fourier series implementation.    
  \begin{figure}[ht!]
\centering
\subfigure[$g$]{\includegraphics[width=0.36\textwidth]{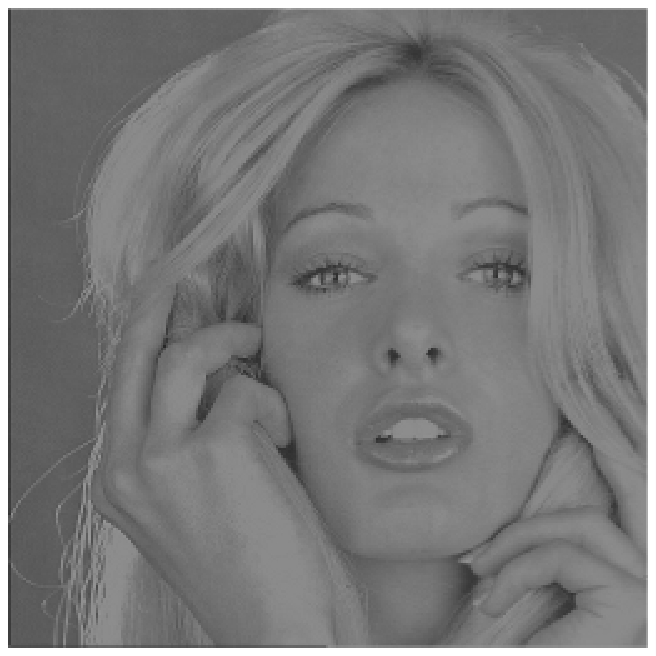}}\quad
\subfigure[$f_{2}$]{\includegraphics[width=0.36\textwidth]{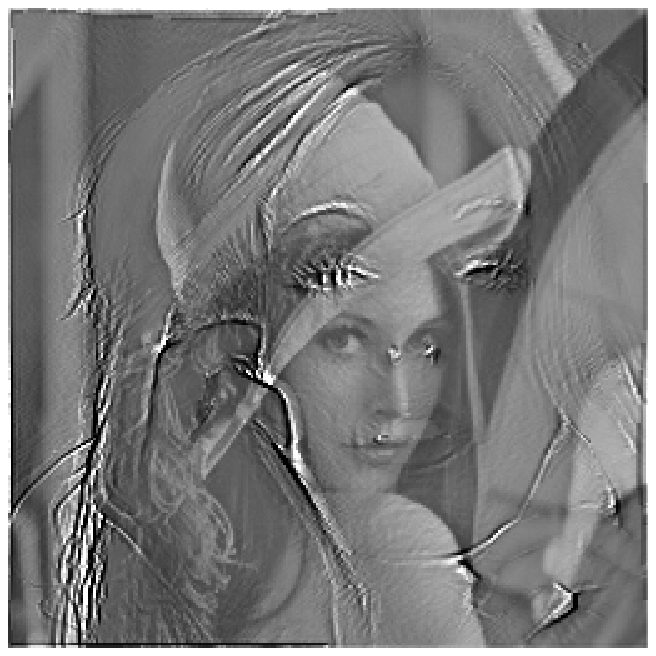}}\\

\subfigure[$f_{5}$]{\includegraphics[width=0.36\textwidth]{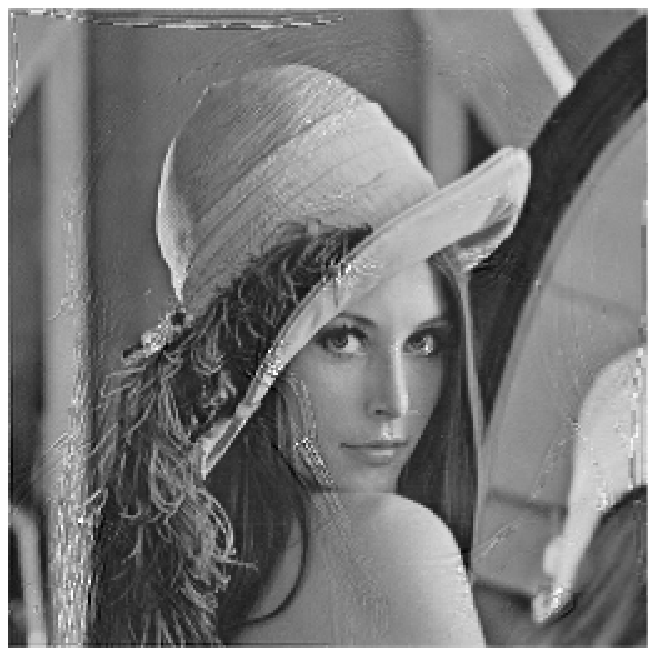}}\quad
\subfigure[$f_{10}$]{\includegraphics[width=0.36\textwidth]{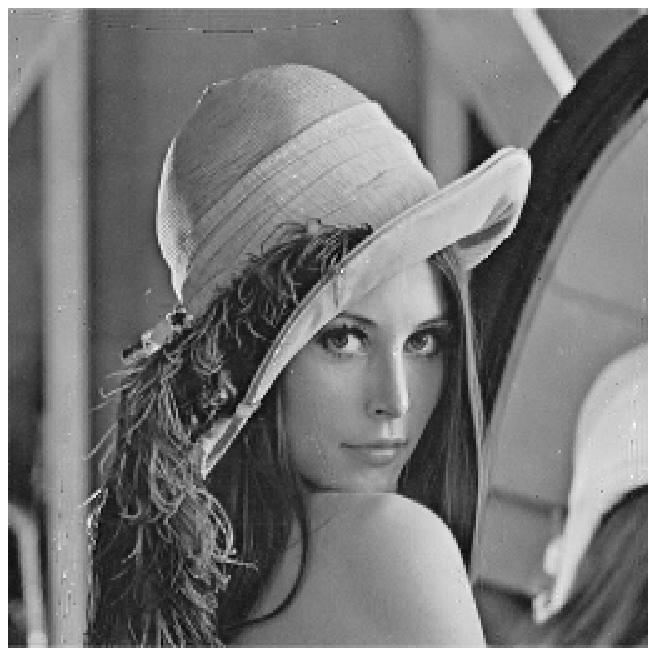}}\\

\subfigure[$f_{20}$]{\includegraphics[width=0.36\textwidth]{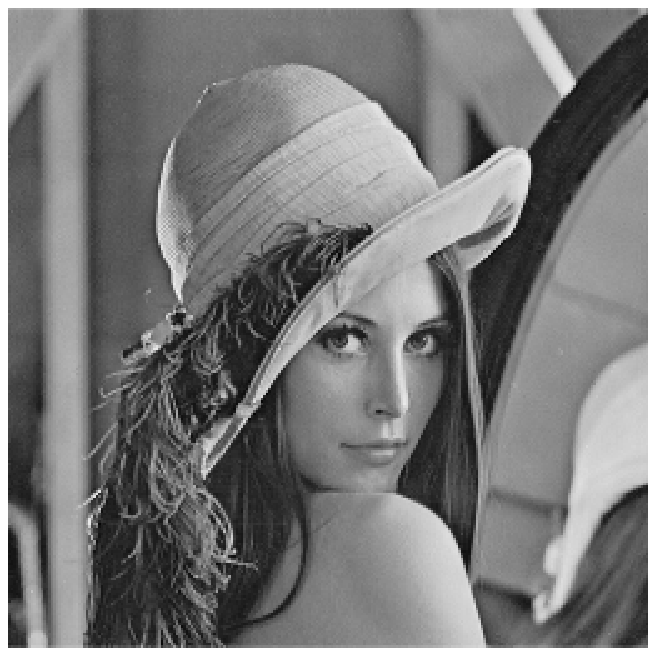}}\quad
\subfigure[$f$]{\includegraphics[width=0.36\textwidth]{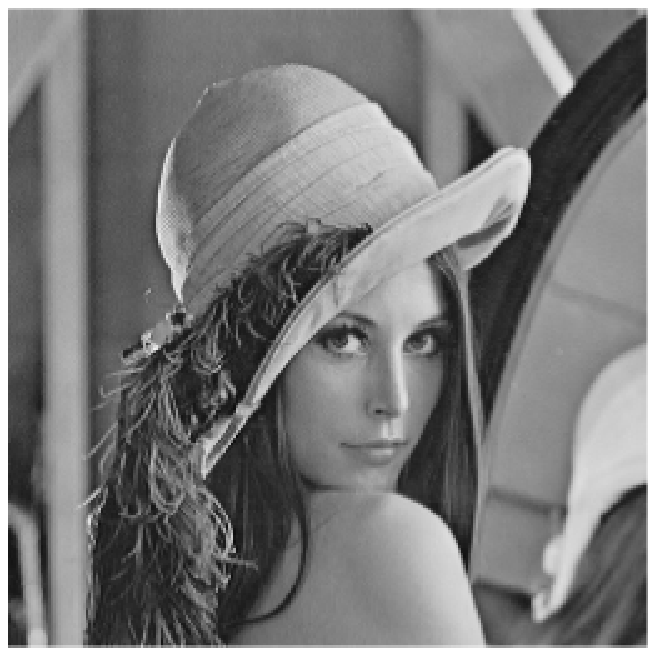}}

\caption{Iterations of the Lena to Tiffany warp for a closest neighbour
interpolation.}\label{lena_tiffany_warp}
\end{figure}  
  The output is presented in Figure \ref{lena_tiffany_warp}. Yet again, the number of
GMRES iterations stayed nearly constant (only about 4 inner iterations) which made
the computing time very small and the spectral radius of the inverse of the
discretization matrix stayed close to $0.03$. One can observe the effect of the
periodic boundary conditions (on $f_{5}$ for example). We see that the periodicity
did not affect the overall outcome. Indeed, $f_{20}$ and $f$ are almost visually
identical. If we use a linear interpolation instead of a closest neighbour
interpolation, we could not visually tell the difference between $f_{20}$ and $f$
since we were able to reduce a lot more the maximum of $f-f_{n}$. This time, our
method did not converge for a $\tau$ equal to 1. This could be explained by
the fact that the images vary a lot and even by translating, some values were still
close to $0$. Moreover, when we repeat the same experiment with the finite
difference implementation, the $\tau$ required jumped to 8 and we had to iterate
about 60 times to get good results. This is yet another argument in favour of the
Fourier transform implementation.

\section{Concluding remarks}\label{sec_conc}

In conclusion, we saw that our algorithm presents a very good way of computing the
numerical solution of the $L^{2}$ optimal transport problem. The Fourier Transform
implementation makes it accurate, fast, stable and thus very efficient. In the
context of image registration, the limitation to densities bounded away from 0 and
to periodic boundary conditions did not seem to be a serious shortcoming for applying
this algorithm to important practical examples. We also saw that even if our method has 
the downside of having to choose a value of $\tau$ without giving too much information 
on how to make this choice a priori, when using the Fourier transform implementation, in all 
our experiments we never had to take a $\tau$ bigger than $2$ to get convergence. We also conducted more
numerical experiments in \cite{Sa}, such as taking the initial and final densities to be the periodic approximation of
gaussian distributions, and the performances were as good as the ones presented here. The interested reader
can consult \cite{Sa} for more details. Furthermore, compared to available numerical methods, our algorithm is very efficient. 
Indeed, Benamou and Brenier's \cite{BB} use of the fluid dynamics reformulation of optimal
transport introduces an extra time variable to the problem which is a non-necessary cost
for all practical purposes of interest to us. In \cite{CW}, Chartrand et al. presented a gradient
descent on the dual problem of the optimal transport problem which produces acceptable results.
However, when they applied their method to the Lena-Tiffany example, numerical
artifacts appeared as they iterated and their method did not fully converge, as
opposed to ours. Finally, Haber et al. introduced a projection algorithm on the mass
preservation constraint in \cite{HRT} which is as of now, probably the best method to
solve the optimal transport problem. Just like the method we developed here, it enjoys
efficiency and stability properties. However our method is
guaranteed to converge in theory (before discretization), thanks to Theorem
\ref{convproof}. This is not necessarily guaranteed in their case.

To pursue this work in the future, the authors would like to extend the method to
different types of boundary conditions. However, for this to happen, we would
require global a priori estimates on the H\"{o}lder norm of the solution of the
Monge-Amp\`{e}re equation, which to the best of our knowledge are not yet available.
Moreover, it would be interesting to implement a parallel version to increase
the performances even more.

\section*{Acknowledgments}
This research is partially supported by Discovery grants and fellowships from the
Natural Sciences and Engineering Research Council of Canada, and by the University
of Victoria.



%% References


%\bibliographystyle{elsarticle-num}

\begin{thebibliography}{10} 



\bibitem{B} {\sc Y. Brenier}, {\em Polar Factorization and Monotone Rearrangements
of Vector Valued Functions}, Comm. Pure Appl. Math. vol. 44 (1991), pp. 375--417.

\bibitem{BB} {\sc J-D. Benamou and Y. Brenier}, {\em A Computational Fluid Mechanics
Solution to the Monge-Kantorovich Mass Transfer Problem}, Numer. Math vol. 84
(2000), pp. 375--393. 

\bibitem{Br} {\sc BrainWeb Simulated Brain Database}, {\em
http://www.bic.mni.mcgill.ca/brainweb/}, McConnell Brain Imaging Centre (BIC) of the
Montreal Neurological Institute, McGill University.

\bibitem{Ca1} {\sc L.  A. Caffarelli}, {\em The regularity of mappings with a convex
potential}, J. Amer. Math. Soc. vol. 5, 1 (1992), pp. 99-104.

\bibitem{Ca2} {\sc L.  A. Caffarelli}, {\em Boundary regularity of maps with  convex
potentials. II}, Ann. of Math. (2) 144, 3 (1996), pp. 453-496.

\bibitem{CE} {\sc D. Cordero-Erausquin}, {\em Sur le transport de mesures
p\'{e}riodiques}, C. R. Acad. Sci. Paris Ser. I 329 (1999), pp. 199--202.

\bibitem{CW} {\sc R. Chartrand, B. Wholberg, K. Vixie and E. Bollt}, {\em A Gradient
Descent Solution to the Monge-Kantorovich problem}, Applied Mathematical Sciences
Vol. 3, no 21-24 (2009), pp. 1071--1080. 

\bibitem{De} {\sc Ph. Delano\"e}, {\em Classical solvability in dimension two of the
second boundary value problem associated with the Monge-Amp\`ere operator}, Ann.
Inst. H. Poincar\'e Anal. Non. Lin\'eaire. 8, 5 (1991), 443-457.

\bibitem{FM} {\sc L. Forzani and D. Maldonado}, {\em Properties of the Solutions to
the Monge-Amp\`{e}re Equation}, Nonlinear Analysis vol. 57, 5-6 (2004), pp.
815--829. 

\bibitem{GL} {\sc G.H. Golub and C.F. Van Loan}, {\em Matrix computations}, JHU Press, Baltimore, USA, 1996.

\bibitem{GT} {\sc D. Gilbarg and N. S. Trudinger}, {\em Elliptic Partial
Differential Equations of Second Order}, Springer-Verlag, Berlin Heidelberg,
Germany, 2001.

\bibitem{HL} {\sc J. B. Hiriart--Urruty and C. Lemar\'{e}chal}, {\em Convex analysis
and minimization algorithms: Fundamentals, Volume 1}, Springer-Verlag, Berlin
Heidelberg, Germany, 1996.

\bibitem{HRT} {\sc E. Haber, T. Rehman and A. Tannenbaum}, {\em An Efficient
Numerical Method for the solution of the $L^{2}$ Optimal Mass Transfer Problem},
SIAM J. Sci. Comput. Vol.32, No.1 (2010), pp. 197--211. 

\bibitem{Kan} {\sc L. Kantorovich}, {\em On the translocation of masses}, Dokl.
Akad. Nauk. USSR 37 (1942), 199-201. English translation in J. Math. Sci. 133, 4
(2006), 1381-1382.

\bibitem{L} {\sc R.J. LeVeque}, {\em Finite difference methods for ordinary and partial differential equations:
steady-state and time-dependent problems}, SIAM, Philadelphia, USA, 2007.

\bibitem{LR} {\sc G. Loeper and F. Rapetti}, {\em Numerical Solution of the
Monge-Amp\`{e}re equation by a Newton's algorithm}, C. R. Acad. Sci. Paris Ser. I
340 (2005), pp. 319--324.

\bibitem{LTW} {\sc J. Liu, N.S. Trudinger and X.J. Wang}, {\em Interior
$\mathcal{C}^{2,\alpha}$ Regularity for Potential Functions in Optimal
Transportation}, Comm. Part. Diff. Eqns. vol. 35 (2010), pp. 165--184, 

\bibitem{M} {\sc G. Monge}, {\em M\'{e}moire sur la th\'{e}orie des d\'{e}blais et
des remblais}, Histoire de l'Acad\'{e}mie Royale des Sciences de Paris, avec les
M\'{e}moires de Math\'{e}matique et de Physique pour la m\^{e}me ann\'{e}e (1781),
p. 666--704.

\bibitem{O} {\sc A. Oberman and B. Froese}, {\em Fast finite difference solvers for singular solutions of the elliptic Monge-Ampère equation},
J. Comput. Phys. Vol. 230 No.3 (2011), pp. 818--834.  

\bibitem{PP} {\sc K.B. Petersen and M.S. Pedersen}, {\em The Matrix Cookbook}, Technical University of Denmark, Copenhagen, Denmark, 2008.

\bibitem{R} {\sc T. Rehman, E. Haber,G. Pryor,J.Melonakos and A.Tannenbaum}, {\em 3D
nonrigid registration via optimal mass transport on the GPU}, Medical Image Analysis
Vol. 13 (2009), pp. 931--940. 

\bibitem{S} {\sc Y. Saad}, {\em Iterative methods for sparse linear systems}, SIAM, Philadelphia, USA, 2003.

\bibitem{Sa} {\sc L.P. Saumier}, {\em An Efficient Numerical Algorithm for the L$^2$ Optimal Transport Problem with Applications to Image Processing},
{\em http://hdl.handle.net/1828/3157}, M.Sc. Thesis, University of Victoria, Canada, 2010.

\bibitem{St} {\sc J. Strain}, {\em Fast Spectrally-Accurate Solution of
Variable-Coefficients Elliptic Problems}, Proc. of the AMS Vol. 122, No. 3 (1994),
pp. 843--850.

\bibitem{TW} {\sc N.S.Trudinger and X.J.Wang}. {\em On the second boundary value
problem for Monge-Amp\`ere type equations and optimal transportation}, Preprint;
http://arxiv.org/abs/math.AP/0601086

\bibitem{Ur} {\sc J. Urbas}, {\em On the second boundary value problem for equations
of Monge-Amp\`ere type}. J. Reine Angew. Math. 487 (1997), pp. 115-124.

\bibitem{V} {\sc C. Villani}, {\em Topics in Optimal Transportation}, Graduate
Studies in Mathematics Vol. 58, AMS, Providence, USA, 2003.

\bibitem{Vi} {\sc C. Villani}, {\em Optimal Transport, Old and New}, Grundlehren der
mathematischen Wissenshaften, Vol. 338, Springer-Verlag Berlin Heidelberg, 2009.

\end{thebibliography}






% \end{thebibliography}

\end{document}